 \documentclass[a4paper,12pt]{article}

\usepackage[latin1]{inputenc}
\usepackage{amsmath}
\usepackage{amsfonts}
\usepackage{amssymb}
\usepackage{amsthm}
\usepackage{indentfirst}
\usepackage{graphicx}
\usepackage{psfrag}

\newtheorem{proposition}{Proposition}[section]
\newtheorem{theorem}[proposition]{Theorem}
\newtheorem{lemma}[proposition]{Lemma}
\newtheorem{corollary}[proposition]{Corollary}

\theoremstyle{definition}

\theoremstyle{definition}

\theoremstyle{definition}

\theoremstyle{definition}
\newtheorem{definition}{Definition}[section]

\newcommand{\mw}{\mathbb{W}}
\newcommand{\md}{\mathop{\mathrm{mod}}}
\newcommand{\Shift}{\mathop{\mathrm{Shift}}}
\newcommand{\floor}[1]{\ensuremath{\lfloor #1 \rfloor}}
\newcommand{\seq}[1]{\ensuremath{\underline{#1}}}
\newcommand{\reg}{\mathrm{\mathrm{Reg}}}
\newcommand{\config}{\mathop{\mathrm{CONF}}}

\newcommand{\rot}{\mathop{\mathrm{Rot}}}

\title{On regular configurations and disjoint cycles in shift graphs}
\author{Taoyang Wu\\
 {\small Department of Computer Science and School of Mathematics
Sciences}\\
{\small Queen Mary, University of London. UK}\\
{\small Taoyang.Wu@dcs.qmul.ac.uk}}

\begin{document}

\maketitle

\begin{abstract}
  Configurations are necklaces with prescribed numbers of red and black beads.
   Among all possible configurations, the regular one plays an important
   role in many applications. In this paper, several aspects of regular configurations are discussed,
   including construction, uniqueness, symmetry group and the
   link with balanced words.

   Another model of configurations is the polygons formed by
    a given number of sides of two different lengths. In this context,
    regular configurations are used
     to obtain a lower bound for the cycles packing number of
    shift graphs, a subclass of the directed
    circulant graphs.
\end{abstract}

\section{Introduction}\label{sec:intro}

Configurations are necklaces with prescribed numbers of red beads
and black beads. More precisely, they are circular arrangements of
a fixed number of red and black beads.  Besides this model,
configurations can also be interpreted in many other ways, such as
finite words in symbolic dynamics~\cite{Lo},
 line drawing in computer graphics~\cite{HP}
  and the Kawasaki-Ising model in statistical
  mechanics.

The main interest of this paper is a class of extremal
configurations which we call regular configurations. Intuitively, in these
 configurations the colors are evenly distributed. Alternatively, they are the
closest configurations to the ``random" ones. Regular
configurations are closely linked with balanced words, a well
studied object in symbolic dynamics.

In this paper we will show that there is a unique regular
configuration (up to cyclic shifts) for a given number of red
beads and black beads. Furthermore, an algorithm for computing
this configuration is presented. Some other properties of regular
configurations are also discussed. To this end we introduce the
concept of dual configurations and prove that a configuration is
regular if and only if its dual configuration is regular.

Another topic in this paper is shift graphs. A shift graph is a
directed Cayley graph of $\mathbb{Z}_n$ with two generators. It is
also called double loop, cyclic graph or chordal ring in the
literature~\cite{GN,LM,BCH,Ma}.  Shift graphs form a subclass of
  circulant graphs, a type of graph that has been
intensively studied~\cite{BT,CGV,ET,LPW,ZYG} and has a vast number
of applications to telecommunication network, VLSI design and
distributed computations~\cite{BCH,Cai,Le,Ma}.

For shift graphs, we use regular configurations to obtain a lower
bound for the cycles packing number, the maximum number of
pairwise vertex disjoint cycles. This result also gives an
extremal result for sequences and is used to study the guessing
number of shift graphs~\cite{CamRiisWu}.

The remainder of this paper is organized as follows. In
Section~\ref{sec:model} we present a few models to represent
configurations.  The definition of regular configuration is given
in Section~\ref{sec:regular}, where we also briefly discuss some
of its properties. The notion of dual configurations is introduced
in Section~\ref{sec:dual}, with which we study several aspects of
regular configurations in Section~\ref{sec:properites}, including
construction, uniqueness and the symmetry group. The link between
regular configurations and balanced words is discussed in
Section~\ref{sec:words}. The application of regular configurations
in shift graphs is investigated in Section~\ref{sect:cycles}.

 \section{Configurations and models}\label{sec:model}

Given $a,b \in \mathbb{N}$, let $\config(a,b)$ denote the set
formed by all configurations with a pair of parameters $(a,b)$.
Intuitively, in {\em necklace model}, $a$ and $b$ are respectively
denoting the number of red beads and black beads. And we will use
$n=a+b$ to denote the length of necklaces. For instance,
Fig~\ref{fig:ball} is a necklace representing a configuration in
$\config(6,4)$. In this section, we are going to introduce several
other models that can represent the configurations in
$\config(a,b)$.

The second model is {\em the Kawasaki-Ising model}. In this model,
a configuration in $\config(a,b)$ is represented by a map $\phi$
from $V(C_n)$ to $\{+,-\}$ such that $|\phi^{-1}(+)|=a$. Here
$C_n$ is the cycle graph on $n$ vertices. Since a necklace in
$\config(a,b)$ can be regarded as a vertex coloring of $C_n$ such
that $a$ vertices are colored with red while the others with
black, by denoting red color by $+$ and black color by $-$ we can
convert a necklace into a map in this model. See
Fig~\ref{fig:Ising} for such an example, which is obtained from
the necklace in Fig~\ref{fig:ball}.

In some context, it is convenient to represent $+$ by $1$ and $-$
by $0$. Then the image of a map $\phi$ in $\config(a,b)$, written
as a sequence  $\phi(0)\phi(1)\cdots\phi(n-1)$, is a word over
$\{0,1\}$.  More precisely, in {\em word model}, a configuration
in $\config(a,b)$ is represented by a word of length $n$ and
weight $a$. Here the weight of a word $w$ is defined to be the
number of $1$s in it.

Another model is \emph{line model}, arising in compute
 graphics to answer the following problem: how to draw a zig-zag line from $(0,0)$ to $(a,b)$
 on the screen to approximate the ``real" line through these two points~\cite{HP}.
 We should notice that the scree is represented by the integer
 lattice $\mathbb{Z}^2$ and one step from $(x,y)$ is either $(x+1,y)$
 or $(x,y+1)$. As in Fig~\ref{fig:line},  each configuration can be represented by
 such a zig-zag line.

The last model we will mention is {\em polygon model}, which plays
an important role in Section~\ref{sect:cycles}. In this model, a
red bead is represented by a type I side, a side of length
$\alpha$, while a black bead by a type II side, a side of length
$\beta$. Here we always assume $\alpha \not = \beta$. Then a
configuration in $\config(a,b)$ is a polygon formed by $a$ type
I sides and $b$ type II sides.  See Fig~\ref{fig:polygon} for
a configuration in
  $\config(6,4)$ with $\alpha=1$ and $\beta=3$.

\begin{figure}[h]

\begin{minipage}[c]{2.5in}
\centering
\includegraphics[height=1.5in]{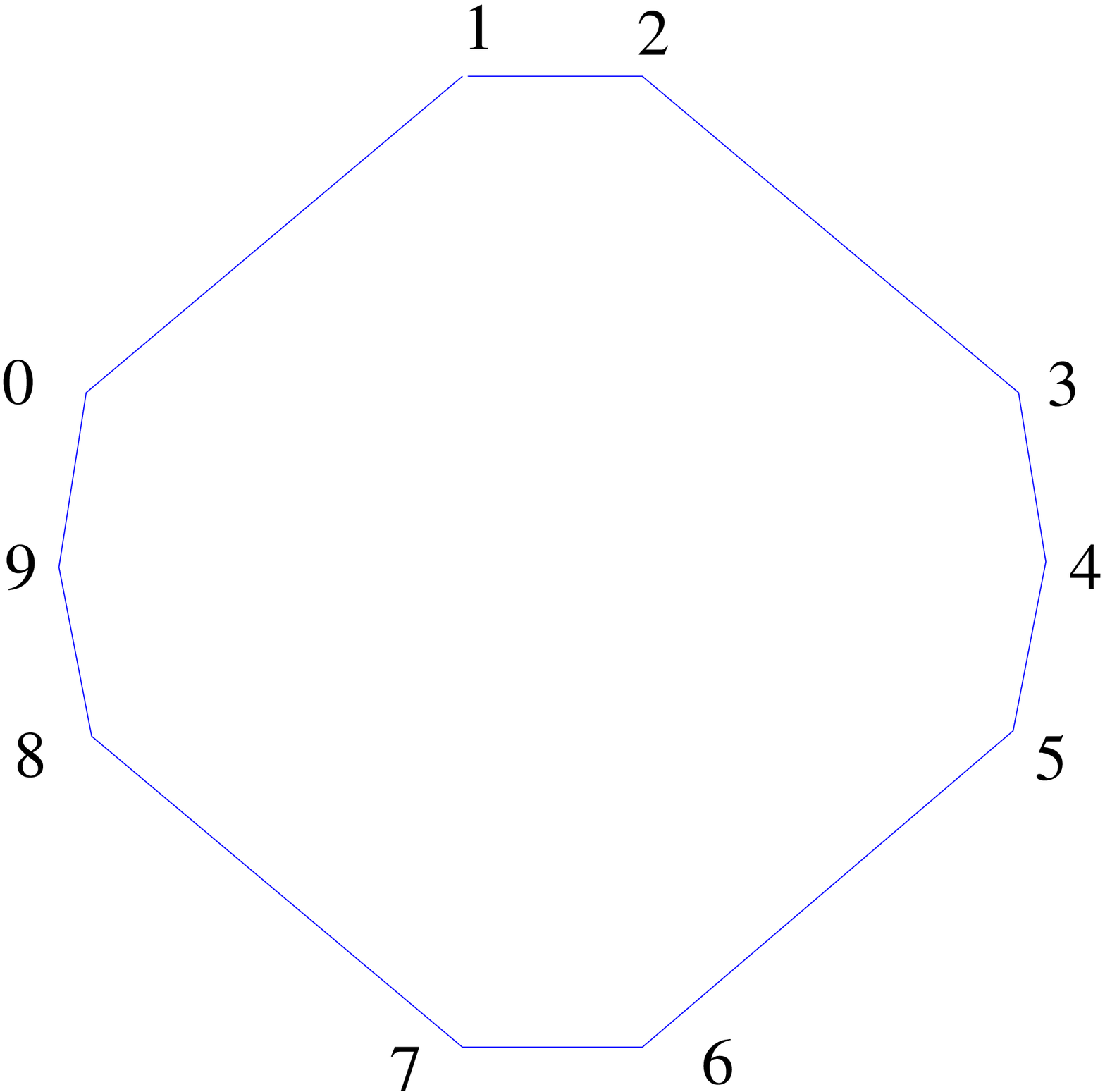}
\caption{Polygon Model}\label{fig:polygon}
\end{minipage}
\hspace{0.2in}
\begin{minipage}[c]{2.5in}
\centering %
\psfrag{B0}[cc][cc]{\small $B_0$}%
\psfrag{B1}[cc][cc]{\small $B_1$}%
\psfrag{B2}[cc][cc]{\small $B_2$}%
\psfrag{B3}[cc][cc]{\small $B_3$}
\includegraphics[height=1.5in]{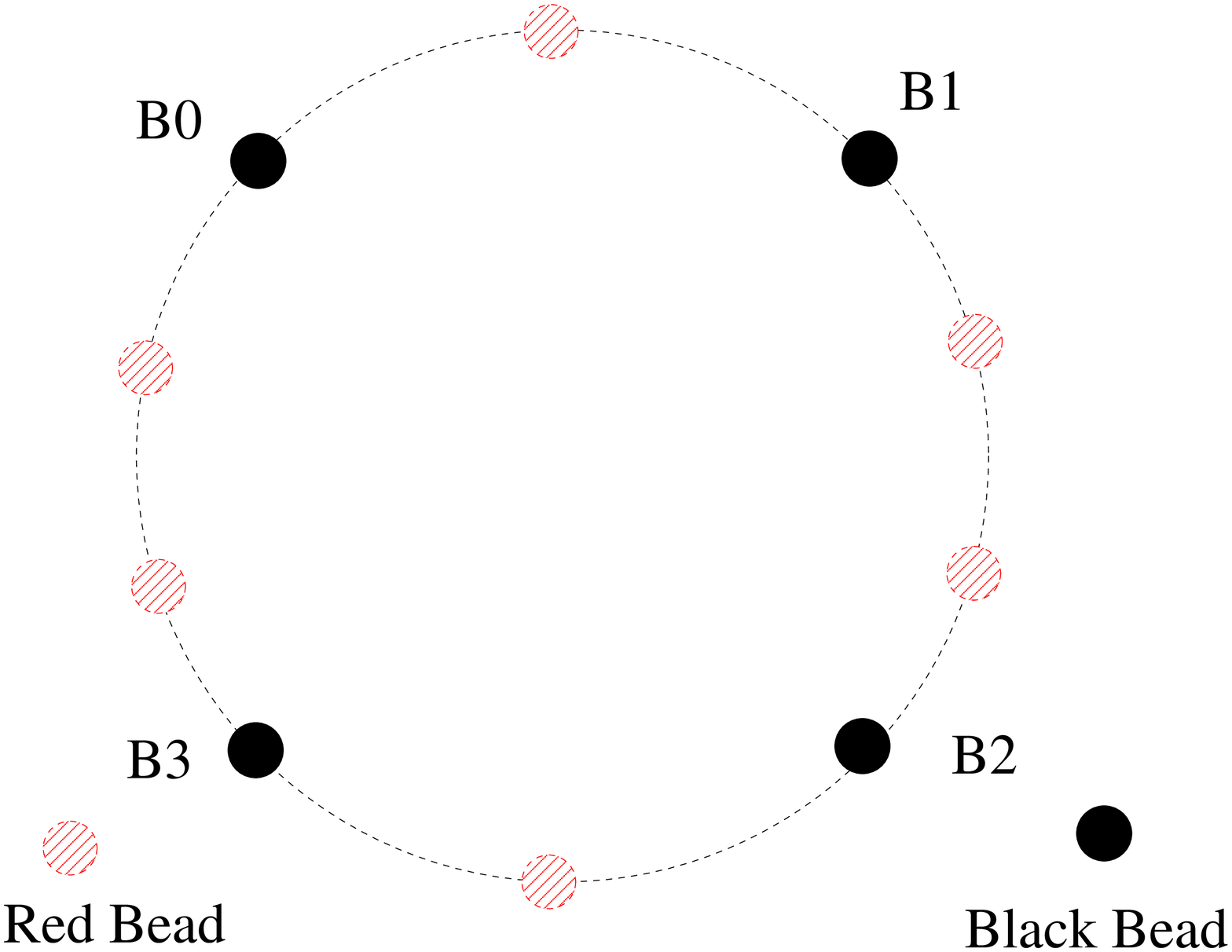}
\caption{Necklace Model}\label{fig:ball}
\end{minipage}

\end{figure}

\begin{figure}[h]

\begin{minipage}[c]{2.5in}
\centering
\includegraphics[height=1.5in]{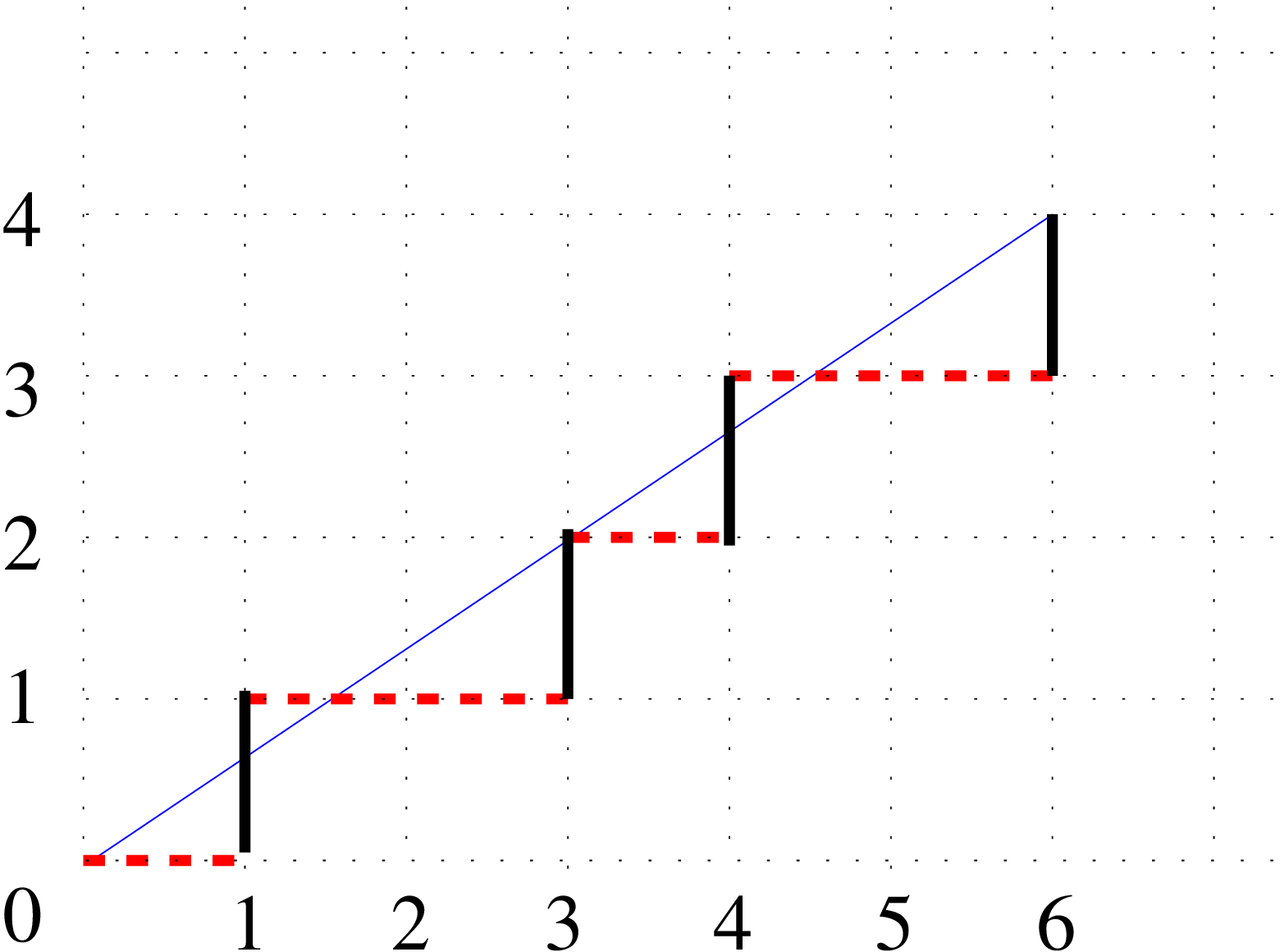}
\caption{Line Model}\label{fig:line}
\end{minipage}
\hspace{0.2in}
\begin{minipage}[r]{2.5in}
\centering
\includegraphics[height=1.5in]{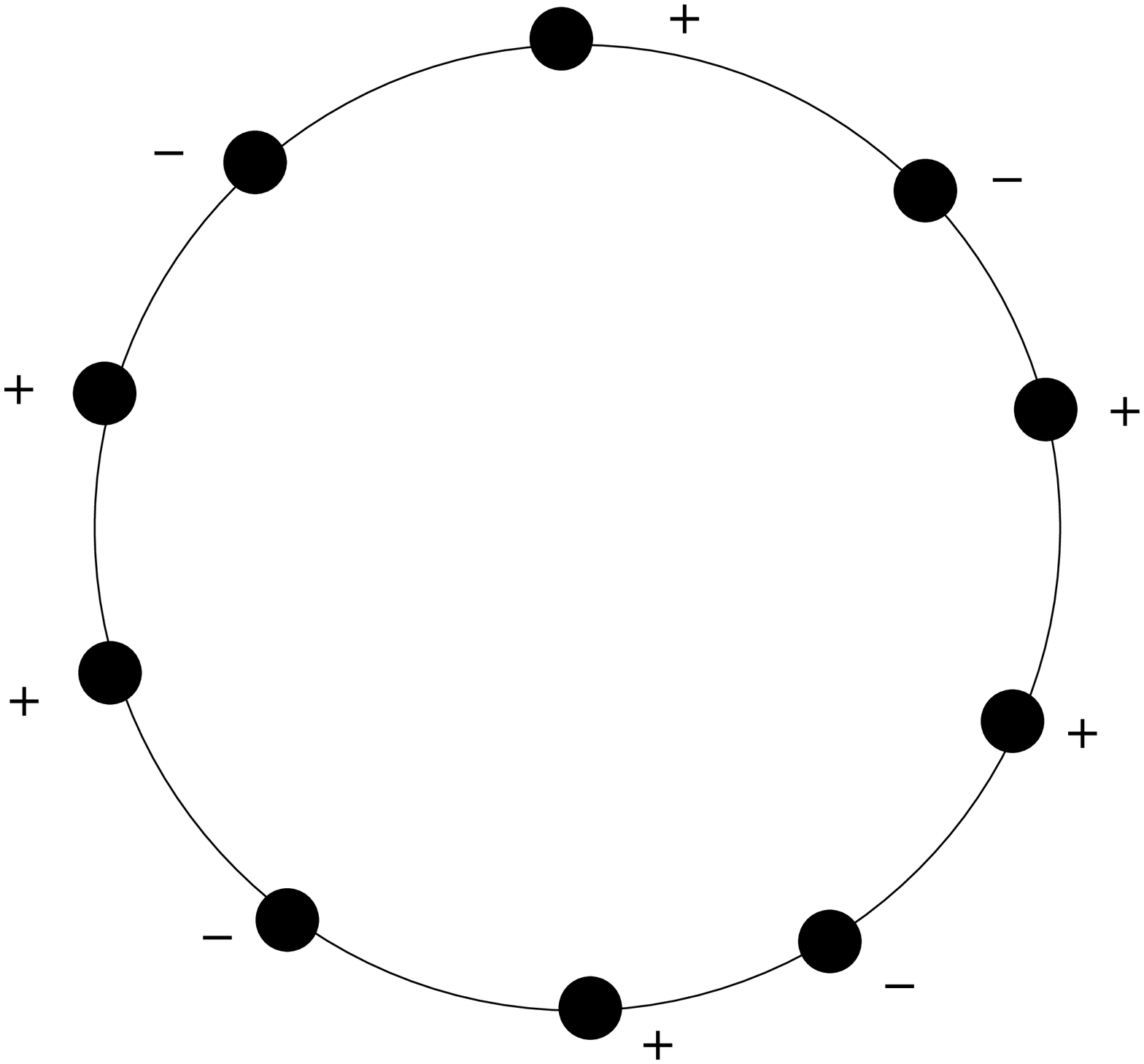}
\caption{Ising Model}\label{fig:Ising}
\end{minipage}

\end{figure}

Here we list all these models for $\config(a,b)$ because each
model relates configurations to different objects. Thus we can
study the structure of $\config(a,b)$ from different views. But in
this paper we are mainly focused on necklace model, with some
applications related with polygon model and word model. In another
 paper~\cite{CamWu}, we study regulations from the view of the Kawasaki-Ising model.

For later use, we will associate two labellings on beads in a
necklace. The first one is to label all beads consecutively from
$0$ to $n-1$. The second one is to consecutively label red(black)
beads from $0$ to $a-1$ ($b-1$). To avoid potential confusing, in
the second labelling the red (black) bead with label $i$ with be
denoted by $R_i$ ($B_i$). For the first kind labelling on
polygons, we assume the vertex is labelled with the number
assigned to its right side.

A necklace can be obtained by putting $b$ black beads in a round
and then inserting a certain amount of red beads between each pair
of consecutive black beads. More precisely, a configuration
$\Delta\in \config(a,b)$ can be represented by a sequence:
\begin{equation}\label{eq:repseq}
\Delta=\{B_0,\underbrace{R,\cdots,R}_{x_0},B_1,\underbrace{R,\cdots,R}_{x_1},
\cdots,B_{b-1},
\underbrace{R,\cdots,R}_{x_{b-1}}\}.
\end{equation}
 where $x_i$ is the number of
red beads between black beads $B_i$ and $B_{i+1}$.  For brevity,
we also say the sequence $\{x_0,\cdots,x_{b-1}\}$ is the
\emph{characteristic sequence} of the configuration $\Delta$. For
instance, $\{1,2,1,2\}$ is the characteristic sequence of the
configuration in Fig~\ref{fig:ball}.

Two necklaces are considered the same if we can cyclicly rotate
one to the other. To formulate this in the sequence level, we need
the following definitions.

\begin{definition}
A shift operator $\sigma$ on a sequence $\{x_0,\cdots,x_{b-1}\}$
is defined to be:
$\sigma\{x_0,\cdots,x_{b-1}\}=\{x_1,\cdots,x_{b-1},x_0\}$.
\end{definition}

\begin{definition}
$\{x_0,\cdots,x_{b-1}\} \sim \{x'_0,\cdots,x'_{b-1}\}$ if and only
if there exists an integer $t \in [0,b-1]$ such that
$\{x'_0,\cdots,x'_{b-1}\}=\sigma^t \{x_0,\cdots,x_{b-1}\}$
 where $\sigma^t$ means applying the operator $\sigma$ on the sequence for $t$
 times.
\end{definition}

In other words, $\sim$ is an equivalence relation. Furthermore,
$\{x_0,\cdots,x_{b-1}\}$ and $\{x'_0,\cdots,x'_{b-1}\}$ are in the
same equivalent class if and only if there exists  an integer $t
\in [0,b-1]$ such that $x_i=x'_{i+t}$ for all $i$. Here the index
of the elements in the sequence is calculated with modulo the
length of the sequence. The same convenience will apply to the
other sequences in this paper.

\begin{proposition}
 $\{x_0,\cdots,x_{b-1}\}$ and
$\{x'_0,\cdots,x'_{b-1}\}$ characterize the same configuration in
$\config(a,b)$ if and only if $\{x_0,\cdots,x_{b-1}\} \sim
\{x'_0,\cdots,x'_{b-1}\}$.
\end{proposition}

The above proposition can be verified directly from the
definitions. Intuitively, it says that a configuration is
characterized by a unique equivalence class of sequences.
Therefore, in the following sections we also use a sequence to
represent a configuration in $\config(a,b)$.

\section{Regular configurations}\label{sec:regular}

Among all possible configurations in $\config(a,b)$, the regular
one plays an important role in many applications. In this section
we will give a precise definition of regularity and study some of
its properties. When $a=0$ or $b=0$, there is only one
configuration in $\config(a,b)$, the necklace formed by all black
beads or by all red beads. To avoid this trivial case, in the
remainder of this paper, we will assume $a>0$ and $b>0$ unless
explicitly stated otherwise.

Given a configuration $\Delta\in \config(a,b)$ with its
characteristic sequence $\{x_0,\cdots,x_{b-1}\}$, the following
equation holds since each red bead should be inserted between a
pair of consecutive black beads.
\begin{equation}\label{eq:total}
x_0+x_1+\cdots+x_{b-1}=a
\end{equation}

Intuitively, when $a=bt$ for some $t\in \mathbb{N}$, the regular
configuration is characterized by the sequence $\{t,t,\cdots,t\}$. In
fact, in this case $t$ is the
 expected number of red beads between each pair of consecutive black beads in a random configuration. By random configuration we means to put the beads independently and randomly along the circle.
  But when $a/b$ is not an integer, we need the following more subtle condition. The characteristic sequence of a
regular configuration optimizes the following
problems:
\begin{equation}\label{exp:optimize}
Min(|(x_i+\cdots+x_{i+k-1})-k\frac{a}{b}|),  \forall i,k.
\end{equation}

The above expressions measure the deviation between two
quantities: the left one is $(x_i+\cdots+x_{i+k-1})$, the number
of red beads between $B_i$ and $B_{i+k}$ in the configuration, and
the right one is the expected number of red beads between $B_i$
and $B_{i+k}$ in a random configuration. The smaller value of this
deviation (or {\em discrepancy} as it is sometimes called) would
imply the configuration is closer to the random one. Here we should notice that the
deviation is measured for all possible $i$ and $k$.

 As $(x_i+\cdots+x_{i+k})$ is always
an integer, the expressions in~(\ref{exp:optimize}) can be
simplified as:
\begin{equation}\label{eq:regular}
\frac{a}{b}k-1<x_i+x_{i+1}+\cdots+x_{i+k-1}<\frac{a}{b}k+1
\end{equation}
for $0\leqslant i \leqslant {b-1}, 1\leqslant k \leqslant
1+\floor{\frac{b}{2}}$. Here we only need to consider the cases
for $1\leqslant k \leqslant 1+\floor{\frac{b}{2}}$ as the other
cases can be reduced from them via Equation (\ref{eq:total}).

Denote the system of inequalities in~(\ref{eq:regular}) by
$\reg(a,b)$. Then a sequence satisfies $\reg(a,b)$ if and only if
all its equivalent sequences satisfy it. With these preparations,
we are ready to present the formal definition of regularity.

\begin{definition}\label{prop:idealregular}
A configuration $\Delta$ in $\config(a,b)$ is regular if its
characteristic sequences satisfy the inequalities $\reg(a,b)$. In
this case, we also say its characteristic sequences are regular.
\end{definition}

Let $\mu_j$ be the minimal number of red beads among $j+1$
consecutive black beads. More precisely, let $\mu_{-1}=\mu_0=0$
and
$$\mu_j=\min_{0\leqslant i \leqslant
b-1}\{x_i+x_{i+1}+\cdots+x_{i+j-1}\}$$ for $0<j\leqslant b$.
 Then the regularity can
be characterized in the following way:
\begin{lemma}\label{lem:regsecond}
$\Delta$ is regular if and only if $1+\mu_j> \frac{a}{b}j$ for
$-1\leqslant j \leqslant b$.
\end{lemma}
\begin{proof}
``$\Rightarrow$" This direction can be verified directly from
the inequalities in~(\ref{eq:regular}). \\
``$\Leftarrow$" In this direction, the left inequalities in
~(\ref{eq:regular}) are easy. From the assumptions and Equation~(\ref{eq:total}), we have:
$$x_i+x_{i+1}+\cdots+x_{i+k-1}\leqslant
a-\mu_{b-k}<a-(\frac{a}{b}(b-k)-1)=1+\frac{a}{b}k$$ for
$0\leqslant i \leqslant {b-1}, 1\leqslant k \leqslant
1+\floor{\frac{b}{2}}$. This completes the right ones.
\end{proof}

The following two propositions  can be
verified directly from the above lemma. Herew we use the fact that
 there exists a unique regular configuration in $\config(a,b)$, which will be proved in Section~\ref{sec:properites}
\begin{proposition}\label{prop:regbase}
When $a=bt$, the regular configuration in $\config(a,b)$ is characterized by the sequence $\{t,t,\cdots,t\}$.
\end{proposition}

\begin{proposition}\label{prop:regbaseTwo}
When $a=bt+1$, the regular configuration in $\config(a,b)$ is characterized by the sequence $\{t+1,t,\cdots,t\}$.
\end{proposition}

Intuitively, inserting an equal amount of red beads between each
pair of consecutive black beads will not affect the regularity of
the original configuration. This can be stated explicitly as the
following proposition.
\begin{proposition}\label{prop:indreg}
Let $a=tb+r$. The regular configuration in $\config(a,b)$ is characterized
 by the sequence $\{x_0,\cdots,x_{b-1}\}$ if and only if
 the
regular configuration in $\config(r,b)$ is characterized by the sequence
$\{x'_0,\cdots,x'_{b-1}\}$ with
$x'_j=x_j-t$.
\end{proposition}
\begin{proof}
We will prove one direction and leave the other as an
 exercise to the readers. From Lemma~{\ref{lem:regsecond}}, we
 have $1+\mu_j>\frac{a}{b}j$, which implies
$$1+\mu'_j=1+\mu_j-tj>\frac{a}{b}j-tj=\frac{bt+r}{b}j-tj=\frac{r}{b}j$$
for $-1\leqslant j \leqslant b$. By Lemma~{\ref{lem:regsecond}},
this means that $\{x'_0,\cdots,x'_{b-1}\}$ is regular.
\end{proof}

\section{Dual configurations}\label{sec:dual}
In this section we will introduce the concept of duality and prove
that one configuration is regular if and only if its dual
configuration is regular.

Let $\Delta$ be a configuration in $\config(a,b)$ characterizing
by $\{x_0,\cdots,x_{b-1}\}$. More precisely, $\Delta$ can be
expressed in the following form:
$$\Delta=\{B_0,\underbrace{R,\cdots,R}_{x_0},B_1,\underbrace{R,\cdots,R}_{x_1},\cdots,B_{b-1},
\underbrace{R,\cdots,R}_{x_{b-1}}\}.$$

From it we can construct a new configuration:
$$\Delta^{*}=\{R^{*}_0,\underbrace{B^{*},\cdots,B^{*}}_{x_0},R^{*}_1,
\underbrace{B^{*},\cdots,B^{*}}_{x_1},\cdots,R^{*}_{b-1},
\underbrace{B^{*},\cdots,B^{*}}_{x_{b-1}}\}.$$

Intuitively, $\Delta^{*}$ is obtained from $\Delta$ by switching
the colors of the beads. More explicitly, $B_i^{*}$, a black bead
in $\Delta^{*}$, is obtained from $R_i$ in $\Delta$ and $R_j^{*}$
is from $B_j$.
 Then $\Delta^{*}$ belongs to $\config(b,a)$ and
it can be expressed in the following way.
$$\Delta^{*}=\{B^{*}_0,\underbrace{R^{*},\cdots,R^{*}}_{y_0},B^{*}_1,
\underbrace{R^{*},\cdots,R^{*}}_{y_1},\cdots,B^{*}_{a-1},
\underbrace{R^{*},\cdots,R^{*}}_{y_{a-1}}\}.$$

 Here $y_i$ in the above representation is the number of red beads between $B^{*}_i$
 and $B^{*}_{i+1}$ in  $\Delta^{*}$, which is equal to the number of black beads between
$R_i$ and $R_{i+1}$ in $\Delta$ from the
 construction. See Fig~\ref{fig:dual} for the
dual configuration
 of that in Fig~\ref{fig:ball}.

\begin{figure}[h]
\centering
\includegraphics[height=1.5in]{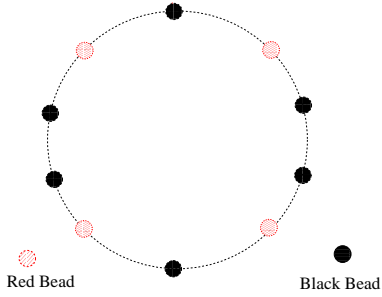}
\caption{A dual configuration}\label{fig:dual}
\end{figure}

Since the color on each bead will remain the same after twice
switching, we have the following proposition.

\begin{proposition}\label{prop:doubledual}
$(\Delta^{*})^{*}=\Delta$.
\end{proposition}

Recall that the regular condition, $\reg(b,a)$, for the
configurations in $\config(b,a)$ is:
\begin{equation}\label{eq:dual}
\frac{b}{a}t-1<y_j+y_{j+1}+\cdots+y_{j+t-1}<\frac{b}{a}t+1
\end{equation}
for all $0\leqslant j \leqslant a-1$ and $1\leqslant t \leqslant
1+\floor{\frac{a}{2}}$.

The following lemma is the main result of this section, which
plays an important role in the remainder of this paper.
\begin{lemma}\label{lem:dual}
A configuration in $\config(a,b)$ is regular if and only if its
dual configuration, which belongs to  $\config(b,a)$, is regular.
\end{lemma}

\begin{proof}
By Lemma~{\ref{prop:doubledual}}, it is enough to show  one
direction.

Given a configuration $\Delta$ in $\config(a,b)$ characterizing by
the sequence $\{x_0,\cdots,x_{b-1}\}$, we need to show
$\{y_0,\cdots,y_{a-1}\}$, which characterizes the dual
configuration $\Delta^{*}$,  satisfies $\reg(b,a)$. That means for
all $0\leqslant j \leqslant a-1, 1\leqslant t \leqslant
1+\floor{\frac{a}{2}}$, the following inequality holds.
\begin{equation}\label{eq:dualproof}
\frac{b}{a}t-1<y_j+y_{j+1}+\cdots+y_{j+t-1}<\frac{b}{a}t+1.
\end{equation}

Without loss of generality, we can prove the above inequalities for $j=1$.  Denote the number of red
beads between $B^{*}_1$ and $B_{t+1}^{*}$ by $\tau$. In
other words, $\tau=y_1+\cdots+y_{t}$. We
will prove the lemma by considering three different cases of $\tau$.\\

{\bf Case 1: $\tau=0$}.~ In this case there is no red bead between
$B^{*}_1$ and $B^{*}_{t+1}$. That means there is no black bead between
$R_1$ and $R_{t+1}$ in $\Delta$, which implies the red beads with
labelling from $R_1$ to $R_{t+1}$ are all falling  between $B_i$ and $B_{i+1}$
for some $i$. Thus $t+1<x_i$ for some $i$, from which we have
\begin{eqnarray*}
 t+1 < x_i\ \mbox{for some} \ i  & \Rightarrow & t+1 <1+\frac{a}{b} \\
    & \Rightarrow & t<\frac{a}{b} \\
    & \Rightarrow & \frac{b}{a}t-1<y_1+y_{2}+\cdots+y_{t}=0<\frac{b}{a}t+1,
\end{eqnarray*}
where the first line comes from the regularity of
$\{x_0,\cdots,x_{b-1}\}$.\\

{\bf Case 2: $\tau=1$}.~ Similar to the proof in case 1, we know
there is only one black bead between $R_1$ and $R_{t+1}$ in $\Delta$,
which means $t+1 < x_i+x_{i+1}$ for some $i$. Thus we have:
\begin{eqnarray*}
 t+1 < x_i+x_{i+1}\ \mbox{for some} \ i & \Rightarrow & t+1 <1+2\frac{a}{b} \\
     & \Rightarrow & \frac{b}{a}t-1<1 \\
    & \Rightarrow &
    \frac{b}{a}t-1<y_1+y_{2}+\cdots+y_{t}=1<\frac{b}{a}t+1.
\end{eqnarray*}

{\bf Case 3: $\tau \ge 2$}.~ Now we can assume that there are
$k+1$ red beads between $B^{*}_1$ and $B^{*}_{t+1}$ for some
$k\geqslant 1$. In other words, there are $k+1$ black beads
between $R_1$ and $R_{t+1}$ in $\Delta$. Assume these black beads are
labelled from $B_i$ to $B_{i+k}$. Then we have the following
fragment in the sequence representing $\Delta$:
$$R_1,\underbrace{R,\cdots,R}_{\epsilon_1}, B_{i}, \cdots ,
B_{i+k}, \underbrace{R,\cdots,R}_{\epsilon_2},R_{t+1}$$

where $0\leqslant \epsilon_1 \leqslant x_{i-1}-1$ and $0\leqslant
\epsilon_2 \leqslant x_{i+k}-1$.\\

Since there are $t-1$ red beads between $R_1$ and $R_{t+1}$, we have:
\begin{equation}{\label{eq:redproof}}
(\epsilon_1+1)+(\epsilon_2+1)+(x_{i}+\cdots+x_{i+k-1})=(t-1)+2=t+1.
\end{equation}

From the regularity of $\Delta$,
$x_i+\cdots+x_{i+k-1}> k\frac{a}{b}-1$. Putting it into
Equation~(\ref{eq:redproof}) and noting that
$\epsilon_1\geqslant0$ and $\epsilon_2\geqslant0$, we have:
\begin{equation}\label{eq:t1}
t+1>k\frac{a}{b}-1+2.
\end{equation}

On the other hand, from $\epsilon_1\leqslant x_{i-1}-1$ and
$\epsilon_2\leqslant x_{i+k}-1$, Equation~(\ref{eq:redproof})
implies:
\begin{equation}\label{eq:t2}
t+1\leqslant x_{i-1}+x_{i}+\cdots+x_{i+k}< \frac{a}{b}(k+2)+1.
\end{equation}

Put (\ref{eq:t1}) and (\ref{eq:t2}) together, we have:
\begin{eqnarray*}
& k\frac{a}{b}-1+2<t+1< \frac{a}{b}(k+2)+1\\
\Rightarrow & k\frac{a}{b}<t<\frac{a}{b}(k+2) \\
\Rightarrow & t\frac{a}{b}-2<k<t\frac{a}{b} \\
\Rightarrow & t\frac{a}{b}-1<k+1<t\frac{a}{b}+1,
\end{eqnarray*}

which completes the proof of the last case in the lemma since
$y_1+y_{2}+\cdots+y_{t}=k+1$.
\end{proof}

To summarize, in this section we prove $*$, the dual operator,
maps each configuration $\Delta$ in $\config(a,b)$ to a
configuration $\Delta^{*}$ in $\config(b,a)$. Furthermore, this
map is an onto bijection and preserves the regularity.

\section{Construction and symmetry}\label{sec:properites}

In this section, two aspects of regular configuration are
discussed. The first one is the existence of a unique regular
configuration in $\config(a,b)$. The other is the symmetry group
of regular configurations.

\subsection{Construction}\label{sect:findingconf}
An algorithm to construct a regular configuration in
$\config(a,b)$ is proposed in this subsection. In fact, there
exists another well known algorithm in computer graphics for this
problem~\cite{HP}. But the one presented here is more convenient
for our propose. Furthermore, we will prove such regular
configuration in $\config(a,b)$ is unique.

The input of the algorithm is $\delta$: $(a,X;~b,Y)$. Here $X$
($Y$) is a fragment of necklaces; $a$ and $b$ are respectively the
number of $X$ fragments and $Y$ fragments. The output is $\Delta$,
a necklace formed by $a$ $X$ fragments and $b$ $Y$ fragments. As
in the previous sections, the output configuration $\Delta$ also
will be represented by its characteristic sequence. \vspace{0.2in}

 \fbox{
\begin{minipage}{5in}

\begin{center}
 {\center FindRegular~$(a,X;~b,Y)$\\}

    \begin{itemize}
       \item if ($a<b$), return FindRegular$(b,Y;~a,X)$;
       \item else do:
          \begin{itemize}
             \item compute $t,k$ such that $a=bt+k$ where $1\leq t,0\leq k < b$.
             \item build a new fragment $Z=\{Y,X,...,X\}$ with $t$ fragments of $X$.
                \begin{itemize}
                       \item if ($k \neq 0$) return FindRegular$(b,Z; k,X)$;
                       \vspace{1mm}
                       \item  else return a necklace formed by $b$ fragments of
                       $Z$.
                \end{itemize}
          \end{itemize}
    \end{itemize}

{\bf Algorithm I: Find a regular configuration}
\end{center}

\end{minipage}
} \vspace{0.2in}

When the fragment $X$ is only a red bead and $Y$ is a black bead,
the input parameters $(a,X;~b,Y)$ can be simplified as $(a,b)$. In
this case, we will prove that the output configuration, which
belongs to $\config(a,b)$, is regular. To this end, we consider some special cases.

\begin{proposition}\label{prop:outputExample}
 For the input $\delta_1=(bt,b)$, the output
configuration $\Delta_1$ is given by the characteristic sequences
    $\{t,\cdots,t\}$. Similarly, for $\delta_2=(bt+1,b)$, $\Delta_2$
    is characterized by $\{t+1,t,\cdots,t\}$. In both cases, the
    output configurations are regular.
\end{proposition}

\begin{proof}
For $i=1,2$, it can be verified directly that $\Delta_i$ is the
output configuration for $\delta_i$. Furthermore $\Delta_i$ is
regular from Proposition~\ref{prop:regbase}
and~\ref{prop:regbaseTwo}.
\end{proof}

\begin{proposition}\label{prop:alginduct}
Let $a=tb+r$ for nonnegative integers $t$ and $r$. If
$\{x_0,\cdots,x_{b-1}\}$ characterizes the output configuration of
FindRegular(a,b), then $\{x'_0,\cdots,x'_{b-1}\}$, where
$x'_j=x_j-t$, characterizes the output configuration of
FindRegular(r,b).
\end{proposition}
\begin{proof}
The proposition can be verified by comparing the outputs of the
algorithm for inputs $(bt+r,b)$ and $(r,b)$.
\end{proof}

\begin{proposition}\label{prop:algdual}
If $\Delta$ is the output of FindRegular(a,b), then $\Delta^{*}$,
the dual configuration of $\Delta$, is the output of
FindRegular(b,a).
\end{proposition}

\begin{proof}
$\Delta^{*}$ can be obtained from $\Delta$ by switching the color
of all beads. This process can be done either before running the
algorithm to get $\Delta$ or after running it. In the first case,
it is the same to say the input is $(b,a)$.
\end{proof}

With these preparations we can prove the following theorem, which
is the main result of this subsection.

\begin{theorem}\label{thm:regFind}
Given any integer pair $(a,b)$ as the input, the output of the
algorithm FindRegular is always a regular configuration.
\end{theorem}

\begin{proof}
We will prove this theorem by induction on $b$.

{\bf Step 1: } The base case is $a=bt$, which also contains $b=1$.
In this case the theorem holds from
Proposition~\ref{prop:outputExample}.

{\bf Step 2: } Now let $(a,b)$ be an instance of the input such
that $b$ is smallest over all instances that the output is a
irregular configuration. From step 1, we can assume $a=tb+r$ for
some integer $t$ and $r$ where $0<r<b$. From
Proposition~\ref{prop:alginduct} and~\ref{prop:indreg} the output
for $(r,b)$ is also irregular. Furthermore, from
Proposition~\ref{prop:algdual} and Lemma~\ref{lem:dual}, the
output for $(b,r)$ is also irregular, which contradicts the
minimality of $b$.
\end{proof}

The above theorem implies the existence of a regular configuration
in $\config(a,b)$. Now we are going to show that such regular
configuration is in fact unique in $\config(a,b)$.

\begin{theorem}\label{thm:regularuniq}
There exists at most one regular configuration in $\config(a,b)$.
\end{theorem}

\begin{proof}
The theorem holds for $b=0$ and $b=1$ since in both cases there is
essentially one configuration and it is regular by the definition.

Now assume the theorem fails for some $b$ and let $b$ be the
smallest one such that $\config(a,b)$ contains two different
regular configurations $\Delta$, $\Delta'$ for some $a$, where
$\Delta$ and $\Delta'$ are characterized respectively by two
non-equivalent sequences $x_0,\cdots,x_{d-1}$ and
$z_0,\cdots,z_{d-1}$. Now we have $a=bt+r$ for some $t\in
\mathbb{N}$ and integer $0\leqslant r<b$. From
proposition~{\ref{prop:indreg}} we know $x'_0,\cdots,x'_{d-1}$ and
$z'_0,\cdots,z'_{d-1}$ represent two different regular
configurations in $\config(r,b)$. That implies $\config(b,r)$
contains two different regular configurations from
Lemma~\ref{lem:dual}, a contradiction to the minimality of $b$.
\end{proof}

We summarize Theorem~\ref{thm:regFind} and~\ref{thm:regularuniq}
as the following one.

\begin{theorem}
There exists a unique regular configuration in $\config(a,b)$,
which can be constructed by Algorithm I.
\end{theorem}

\subsection{The symmetry group}\label{sec:symmtry}

In this subsection we assume the necklaces in  $\config(a,b)$ is
given with the labelling of the first type. That means the beads
are consecutively labelled from $0$ to $n-1$ with $n=a+b$. In this
case, we also call such a necklace, together with its labelling,
as a {\em labelled necklace}. Furthermore, the bead in a labelled
necklace $\Delta$ will be denoted by $t_i$.

Now two labelled necklaces are essentially same if we can
cyclically permutate one to the other. In other words, they
correspond to the same (unlabelled) necklace. More precisely, we
have the following definition.

\begin{definition}\label{def:rot}
Given an integer $k\in[0,n-1]$, the rotation $\phi_k$, which is
defined as $\phi_k(i)=i+k ~(\md~ n)$ for $i\in [0,n-1]$, is
 called a cyclic permutation of the labelled necklace
$\Delta$ in $\config(a,b)$ if $t_i$ and $t_{\phi_k(i)}$ have the same color for each $i$.
\end{definition}

All cyclically permutations of $\Delta$ form a group, called the
\emph{symmetry group} of $\Delta$ and denoted by $\rot(\Delta)$.
Notice that two labelled necklaces $\Delta$ and $\Delta'$ are the
same if $t_i=t'_i$ for each $i$.

\begin{proposition}\label{prop:labelConf}
There are exactly  $(a+b)/|\rot(\Delta)|$ different labelled
configurations associated with the same unlabelled configuration
$\Delta$.
\end{proposition}

\begin{proof}
Given an unlabelled necklace $\Delta$, we can assign it
 with a first type labelling and denote such a labelled necklace as $\Delta_0$.
  From this we can obtain a set of labelled necklaces $\{\Delta_0,\cdots,\Delta_{n-1}\}$,
   via $t^i_0=t^0_i$. That is, we build $\Delta_i$ by assigning $0$ to $t_i$
   in $\Delta_0$. Now  $\Delta_i$ and $\Delta_j$ are the same if and only if there exists an element $\phi\in\rot(\Delta)$ such that $\Delta_i=\phi(\Delta_j)$. Since any labelled necklace obtaind from $\Delta$ must equal to some $\phi_i$, we know there are exactly  $(a+b)/|\rot(\Delta)|$
different labelled configurations corresponding to $\Delta$.
\end{proof}

Intuitively, a configuration $\Delta\in \config(a,b)$ is symmetric
if its symmetry group $\rot(\Delta)$ has the maximal size over all
possible labelled configurations. Note that this maximal number is
bounded above by $\gcd(a,b)$, the greatest common divisor of $a$
and $b$. This is because each element in $\rot(\Delta)$ induces
two cyclic permutations, one for a labelled necklace in
$\config(a,0)$ and the other for that in $\config(0,b)$.

\begin{definition}
A configuration $\Delta$ is symmetric if its symmetry group
$\rot(\Delta)$ has size $\gcd(a,b)$.
\end{definition}

\begin{proposition}\label{prop:inductionregular}
Let $a=tb+r$ for nonnegative integers $t$ and $r<b$. Given a
configuration $\Delta$ in $\config(a,b)$ characterized by $\{x_0,\cdots,x_{b-1}\}$, let $s$ be the smallest number of $x_i$.
Then $\{x'_0,\cdots,x'_{b-1}\}$, where $x'_j=x_j-s$, characterizes
a configuration $\Delta'$ in
$\config(a-sb,b)$. Furthermore, their symmetry groups have the same size. That is, $|\rot(\Delta)|=|\rot(\Delta')|$.
\end{proposition}
\begin{proof}
The proposition holds because inserting $s$ red beads for each consecutive pairs of black beads
will extend a cyclic permutation of $\Delta'$ to that of $\Delta$. And all cyclic permutation in $\rot(\Delta)$ can be obtained by this way.
\end{proof}

When $\Delta$ is the regular configuration in $\config(a,b)$, from the definition of regularity we
know that the $s$, which is defined in the above proposition, is equal to $t$. In this case,
the following corollary holds.

\begin{corollary}\label{cor:indsym}
Let $a=tb+r$ for nonnegative integers $t$ and $r<b$. Given the
regular configuration $\Delta=\{x_0,\cdots,x_{b-1}\}$ in
$\config(a,b)$, we can construct a configuration
$\Delta'$ in $\config(r,b)$ characterized by $\{x'_0,\cdots,x'_{b-1}\}$, where
$x'_j=x_j-t$. Then $\Delta'$ is regular and $|\rot(\Delta)|=|\rot(\Delta')|$.
\end{corollary}

Unlike regular configurations, symmetric configurations in
$\config(a,b)$ are not unique. For instance, Fig~\ref{fig:polygon}
and Fig~\ref{fig:sym} show two symmetric configurations in
$\config(6,4)$. See Fig~\ref{fig:nonsym} for  an example of
nonsymmetric configuration. Here we represent the configurations
in the polygon model for better visualization.
 Therefore, symmetry generally does not imply
regularity. But the converse is true, as the following theorem
implies.

\begin{figure}[h]
\centering %
\begin{minipage}[c]{2.5in}
\centering
\includegraphics[height=1.5in]{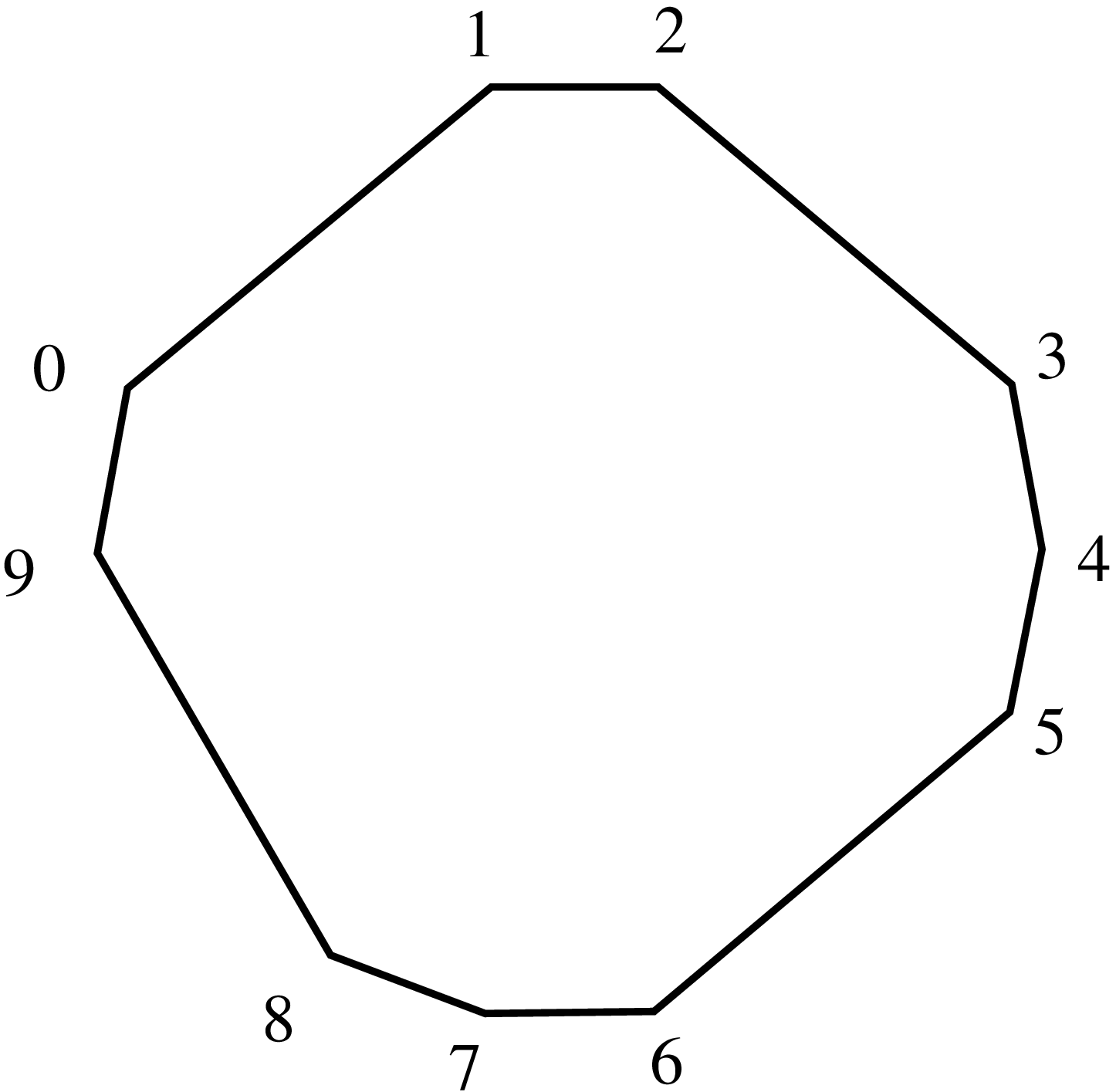}
\caption{Not symmetric}\label{fig:nonsym}
\end{minipage}
\hspace{0.1in}
\begin{minipage}[c]{2.5in}
\centering
\includegraphics[height=1.5in]{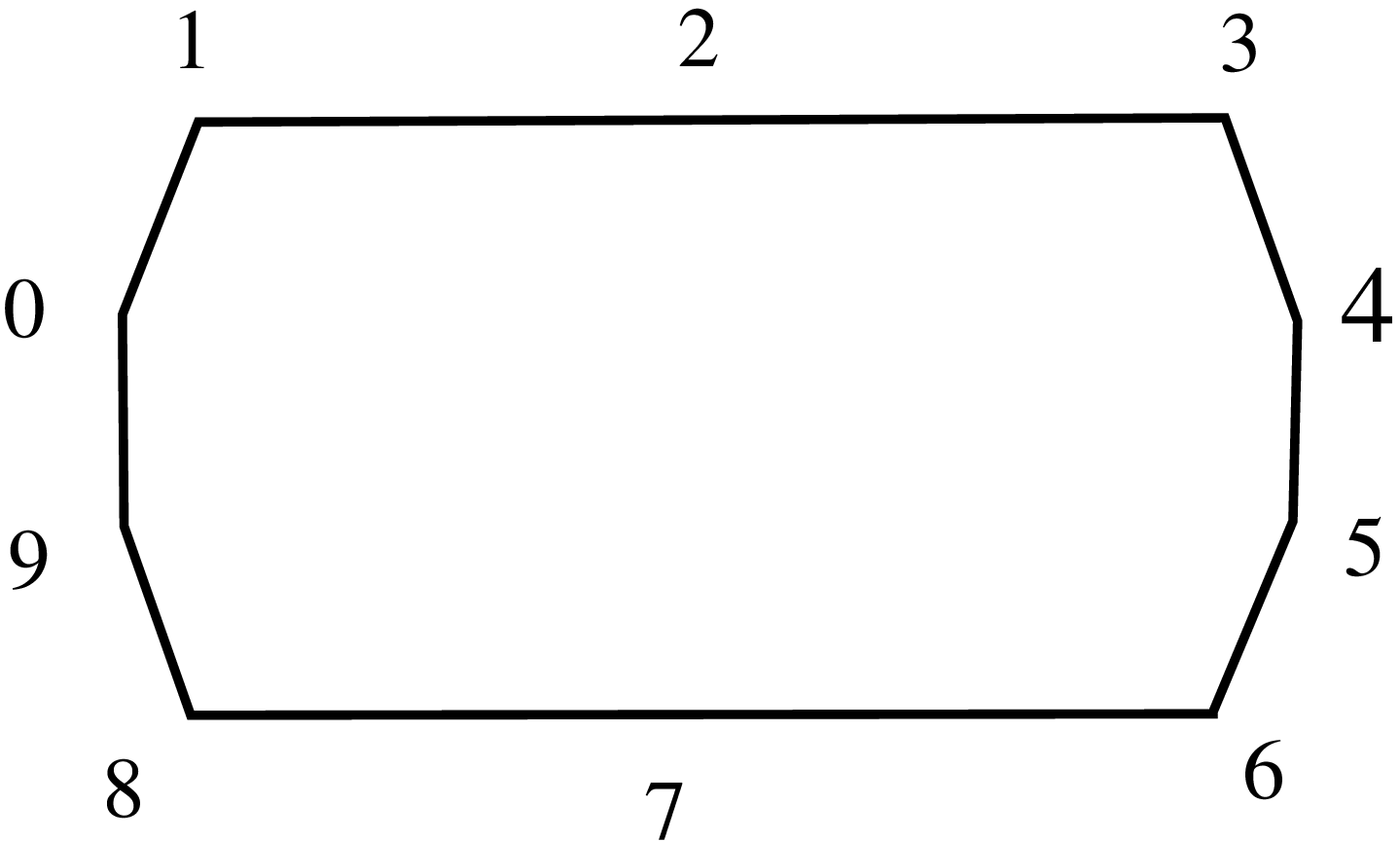}
\caption{Symmetric but not regular}\label{fig:sym}
\end{minipage}
\end{figure}

\begin{theorem}
Regular configurations are symmetric.
\end{theorem}

\begin{proof}
Given the regular configuration $\Delta\in\config(a,b)$, we prove
 $|\rot(\Delta)|=\gcd(a,b)$ by an induction on
$b$.

{\bf Step 1:} The base case is $a=bt$, which includes $b=1$. In
this case the regular configuration $\Delta$ in $\config(a,b)$ is
symmetric as $|\rot(\Delta)|=t$ from Proposition~\ref{prop:outputExample} and $\gcd(a,b)=t$.

{\bf Step 2:} Now assume the theorem fails for some $\config(a,b)$
and let $b$ be the smallest one such that the regular
configuration $\Delta$ in $\config(a,b)$ satisfies
$|\rot(\Delta)|<\gcd(a,b)$. From step 1, we can assume $a=bt+r$
for some $t\in \mathbb{N}$ and $r\in (0,b)$. Now consider
$\Delta'$, the regular configuration in $\config(r,b)$, which is
characterized by $(x'_0,\cdots,x'_{d-1})$ by
Proposition~\ref{prop:indreg}. This implies
$|\rot(\Delta)|=|\rot(\Delta')|$ from Corollary~\ref{cor:indsym}.
Since $\gcd(a,b)=\gcd(r,b)$, $\Delta'$ is not symmetric, a
contradiction to the minimality of $b$.
\end{proof}

\section{Balanced words}\label{sec:words}

In this section we will discuss the relations between balanced
words and regular configurations. Let
$\omega=\omega_0\omega_1\cdots \omega_{n-1}\in\{0,1\}^{n}$ be a
word of length $n$ over alphabet $\{0,1\}$.  The weight of
$\omega$, denoted by $|\omega|_1$, is the number of 1s appeared in
$\omega$. All words  of length $n$ and weight $k$, where $0\leq k
\leq n$, form a set $\mw_{k,n}$.

An operator $\sigma$, called {\em cyclic shift}, is defined on
$\mw_{k,n}$ as follows:  $\sigma(w)=w_1\cdots w_{n-1} w_0$. This
gives an equivalence relation on $\mw_{k,n}$: $w\sim w'$ if and
only if they belong to the same cyclic shifting orbit. Here the
\emph{cyclic shifting orbit} of a word $w$ is defined to be
$\{\sigma^i(w) ~|~ 0\leq i <a+b \}$.

 From
Section~\ref{sec:model}, a word $w\in \mathbb{W}_{a,a+b}$ can be
regarded as a labelled configuration $\Delta \in \config(a,b)$.
From Proposition~\ref{prop:labelConf}, each (unlabelled) necklace
$\Delta$ can be associated with $(a+b)/|\rot(\Delta)|$ different
words, which form an orbit of the cyclic shifting. More precisely,
we have the following relation between words and configurations:

$$\frac{\mathbb{W}_{a,a+b}}{\sim} \cong \config(a,b).$$

 In the remainder of this section, we are going to
show that regular configurations are related to balanced words, an
important class of words.

Let $|w|$ denote the length of $w$ and define $|w|_0$ as
$|w|-|w|_1$. A {\em cyclic subword} of $w$ is any length-$q$
prefix of some $\sigma^{i-1}(w)$ for $1\leq i, q \leq m$. Then we
have the following definition.

\begin{definition}
A word $w$ is called {\em balanced} if for any two of its cyclic
subwords $z$ and $z'$, $|z|=|z'|$ implies $||z|_i-|z'|_i|\leq 1$
for $i=0,1$.
\end{definition}

 Balanced words, the finite version of sturmian words, form an important class of words.
We recommend~\cite{Lo} for more backgrounds and
~\cite{Jenkinson,JZ} for some recent developments.

\begin{theorem}
A configuration $\Delta \in \config(a,b)$ is regular if and only
if any of its associated word $w$ is balanced.
\end{theorem}

\begin{proof}

By definition, a word $w$ is balanced if and only if any word in
its cyclic shifting orbit is balanced. Thus we can always choose a
convenient one in the orbit corresponding to $\Delta$ to simplify
the following proof.  Without loss of generality, we can also
assume $a\geq b$.

$``\Rightarrow"$: If $w$ is not balanced, then $| |u|_i-|v|_i |
\geq 2$ for a pair of cyclic subwords $u$ and $v$ with the same
length, say $t$. Without loss of generality, we can assume
$|u|_0-|v|_0\geq 2$. Furthermore, we associate a second kind
labelling on black beads, i.e., the $0$s in the word. Assume the
first 0 appeared in $u$ is labelled with $1$. Then the structure
of $u$ can be schematically represented in the following way.

$$u=\underbrace{1,\cdots,1}_{\epsilon_1}, 0_{1}, \cdots ,
0_{k}, \underbrace{1,\cdots,1}_{\epsilon_2}$$ where $k$ is the
number of $0$s appeared in $u$, $0 \leq \epsilon_1\leq x_{0}$ and
$0 \leq \epsilon_2 \leq x_{k}$. Here $x_i$ is the number of $1$s
(red beads) appeared between $0_i$ ($B_i$) and $0_{i+1}$
$B_{i+1}$. Since $|u|=t$, these parameters satisfy the following
equation.
\begin{equation}\label{eq:u}
\epsilon_1+x_1+\cdots+x_{k-1}+\epsilon_2+k=t.
\end{equation}

On the other hand, we have the following representation of $v$.

$$v=\underbrace{1,\cdots,1}_{\epsilon'_1}, 0_{i}, \cdots ,
0_{s+i-1}, \underbrace{1,\cdots,1}_{\epsilon'_2}.$$ where $s$ is
the number of $0$s appeared in $v$, $0 \leq \epsilon'_1\leq
x_{i-1}$ and $0 \leq \epsilon'_2 \leq x_{i+s-1}$. Similar to
Equation~(\ref{eq:u}), they satisfy:
\begin{equation}\label{eq:v}
\epsilon'_1+x_i+\cdots+x_{i+s-2}+\epsilon'_2+s=t.
\end{equation}

Since we can solve $k$ and $s$ from Equation~(\ref{eq:u})
and~(\ref{eq:v}) respectively, the condition $k-s \geq 2$, which
comes from $|u|_0-|v|_0 \geq 2$, becomes:
\begin{equation}
(\epsilon'_1+x_i+\cdots+x_{i+s-2}+\epsilon'_2) -
(\epsilon_1+x_1+\cdots+x_{k-1}+\epsilon_2) \geq 2.
\end{equation}

By the constraints of $\epsilon$ and $\epsilon'$, the above
equation can be further simplified as:
\begin{equation}\label{eq:uv}
(x_{i-1}+x_i+\cdots+x_{i+s-1}) - (x_1+\cdots+x_{k-1}) \geq 2.
\end{equation}

But from the condition that $\Delta$ is regular, we can obtain an
upper bound of the sum in the first parenthesis and a lower bound
for that in the second one.

\begin{equation}
(x_{i-1}+x_i+\cdots+x_{i+s-1})< (s+1)\frac{a}{b}+1.
\end{equation}

\begin{equation}
(x_1+\cdots+x_{k-1})>(k-1)\frac{a}{b}-1.
\end{equation}

The above two bounds give us the following inequality:
\begin{eqnarray*}
(x_{i-1}+\cdots+x_{i+s-1}) - (x_1+\cdots+x_{k-1}) &<&
(s+1)\frac{a}{b}+1 - [(k-1)\frac{a}{b}-1]\\
 &=& (s+2-k)\frac{a}{b}+2\\
 &\leq& 2.
\end{eqnarray*}
That is a contradiction to Equation~(\ref{eq:uv}). In the last step of above inequalities
we use the fact that $s+2 \leq k$. \\

$``\Leftarrow"$: In this direction, we need to prove that $\Delta$
is regular with the assumption that $w$, one of its associated
words, is balanced. If this fails, then we have:
\begin{equation}
|x_i+\cdots+x_{i+k-1} - k\frac{a}{b}| \geq 1
\end{equation}
 for some $i$ and $k$. Among all such pairs $(i,k)$ satisfied the
 above inequality, we fix one pair $(i,k)$ such that $k$ is the minimal.
 That means either $x_i+\cdots+x_{i+k-1}\geq 1+(ka/b)$ or
$x_i+\cdots+x_{i+k-1}\leq -1+(ka/b)$. Here we only prove this
direction for the first case as the following arguments can be
easily modified for the second one.

Firstly we claim there exists one $j$ such that
$x_j+\cdots+x_{j+k-1}< ka/b$. If not, then
$$k(x_0+\cdots+x_{b-1})=\sum_{s=0}^{b-1}
(x_s+\cdots+x_{s+k-1})\geq ka+1,$$ a contradiction to the fact
$(x_0+\cdots+x_{b-1})=a$. Thus we have
\begin{equation}
(x_j+\cdots+x_{j+k-1})-(x_i+\cdots+x_{i+k-1})\geq 2
\end{equation}
because both sums in the parentheses are integer.

Now there exist the following two fragments in the configuration
(Recall that 1 stands for read bead $R$ and 0 stands for black
bead $B$):

$$u=0_{i},\underbrace{1,\cdots,1}_{x_i}, 0_{i+1}, \cdots ,
0_{i+k-1}, \underbrace{1,\cdots,1}_{x_{i+k-1}}, 0_{i+k}$$

and

$$v=0_{j},\underbrace{1,\cdots,1}_{x_j}, 0_{j+1}, \cdots ,
0_{j+k-1}, \underbrace{1,\cdots,1}_{x_{j+k-1}}, 0_{j+k}.$$

Furthermore, construct a new fragment $v'$ by choosing the first
$|u|+1$ bits from $v$ and deleting $0_j$. Then $|u|=|v'|$ and
there exists at least two more 0s in $u$ than that in $v'$ since
the 0s in $u$ are labelled from $i$ to $i+k$ while the labels of
0s in $v'$ are falling into the interval $[j+1,j+k-1]$. In other
words, $| |u|_1-|v'|_1| \geq 2$. As each fragment can be realized
as a cyclic subword of $w$, $u$ and $v'$ are two cyclic subwords
with the same length but their weight are different greater than
2, a contradiction to the fact that $w$ is balanced.
\end{proof}

From the relation between words and configurations, the above theorem implies the following corollary.
\begin{corollary}
$\mathbb{W}_{a,a+b}$  has exactly ${(a+b)}/{\gcd(a,b)}$ balanced
words, which form a cyclic shifting orbit that corresponds to the
regular configuration in $\config(a,b)$.
\end{corollary}

 It is already know in~\cite{BS,JZ} that there are precisely $a+b$ balanced
words in $\mathbb{W}_{a+b}$ if $a$ and $b$ are coprime. The above corollary slightly
 generalizes that result. From Section~\ref{sec:symmtry} we also know that the orbit
  corresponding to the regular
  configuration should has the smallest size.

\section{Disjoint cycles in shift graphs}\label{sect:cycles}

In this section we will study the cycles packing number of
$\Shift(n,m)$, the directed Cayley graph of $\mathbb{Z}_n$ with
two generators $\{1,m\}$.

Recall that the vertex set of $\Shift(n,m)$ is
$\{0,1,2,\cdots,n-1\}$ and there are two types of edge sets: type
I consists of the edge generating by $\{1\}$, i.e., the edge has
the form $(i,i+1)\ (\md n)$; type II consists of the edge
generating by $\{m\}$, i.e., the edge has the form
 $(i,i+m)\ (\md n)$. Here $i$ runs through all vertices.
 See Fig~\ref{fig:shift} for $\Shift(9,3)$.

\begin{figure}[h]
\centering
\includegraphics[height=1.5in]{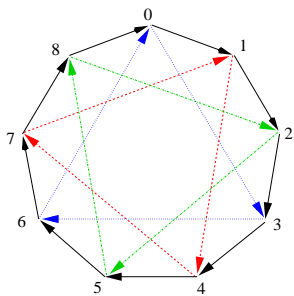}
 \caption{Shift (9,3) }\label{fig:shift}
\end{figure}

The cycles discussed in this paper are directed. A cycle
$C=v_0v_1\cdots v_{d-1}v_0$ can be represented by its vertex
sequence: $A_C=<v_0,v_1,\cdots,v_{d-1}>$ where $d$ is the size of
$C$.  On the other hand, the vertices of a cycle $C$ is the
unordered $d$-set: $V(C)=\{v_0,v_1,\cdots,v_{d-1}\}$.

Two cycles $C$ and $C'$ are called (vertex) disjoint if $V(C)\cap
V(C')=\emptyset$. A collection of disjoint cycles is a set of
disjoint cycles $\mathcal{C}=\{C_1,\cdots, C_k\}$ such that they
are pairwise disjoint. The size of a collection $\mathcal{C}$, is
the number of cycles it contains and will be denoted by
$|\mathcal{C}|$.

\begin{definition}
The cycles packing number for a graph $G$, denoted by $\nu_0(G)$,
is defined as: $$\nu_0(G)=\max
\{|\mathcal{C}|~|~\mbox{$\mathcal{C}$ is a collection of vertex
disjoint cycles in $G$}\}.$$
\end{definition}

Another version of cycles packing number is studying edge disjoint
cycles but in this paper we are only considering the vertex
disjoint version. We remark that $\nu_0(G)$ plays an important
role in many fields~\cite{BG}. One recent example is that
$\nu_0(G)$ gives the lower bound of the guessing number of $G$, a
parameter of graph defined by Riis~\cite{CamRiisWu,Riis}.

Let $n=am+b$ for $0\leq a$ and $0\leq b < m$. Let $d=a+b$ and
$k=\floor{n/(a+b)}$ in the remainder of this section. Then we have
the following theorem, which is the main result of this section.

\begin{theorem}\label{thm:main2}
$\nu_0(\Shift(n,m) )\geq k$ . Furthermore, there is an algorithm
to produce a collection of disjoint cycles $\mathcal{C}$ with
$|\mathcal{C}|=k$.
\end{theorem}

 Before proving the theorem, we use $\Shift(9,3)$ (see Fig~\ref{fig:shift})
 to illustrate
the intuition behind it. Let $C_1=\{0,3,6\}$, $C_2=\{1,4,7\}$ and
$C_3=\{2,5,8\}$. Then  $\nu_0(\Shift(9,3))=3$
and $\mathcal{C}=\{C_1, C_2, C_3\}$ is such a collection. Furthermore, $C_2$ and $C_3$ can be regarded as
being obtained from $C_1$ via a ``rotation". Here one crucial
observation is that the ``polygon'' under $C_1$ is regular.
But we need to consider regular
configurations for general cases.

To this end, associate a necklace to each cycle in $\Shift(n,m)$.
Given a cycle $C$  by its vertex sequence
$A=<v_0,v_1,\cdots,v_{d-1}>$. Then its {\em differential sequence}
is defined as $<v_1-v_0,\cdots,v_{d-1}-v_{d-2},v_0-v_{d-1}>$ and
denoted by $\nabla(A)$. Here the substraction is calculated with
modulo $n$.

For a cycle $C$ in $\Shift(n,m)$, the sequence $\nabla(A_C)$
consists of two numbers, $1$ and $m$. Furthermore, they are
respectively corresponding to two types of edges in $\Shift(n,m)$.
If we represent $1$ in $\nabla(A_C)$ by black bead $B$ and $m$ by
red bead $B$, then $\nabla(A_C)$ gives us a labelled necklace in a
natural way. By forgetting its labelling, we obtain a
configuration from $C$, which will be denoted by $\Delta_C$.

On the other hand, given a configuration $\Delta$ and a vertex
$v$, we can construct a path $C_{v,\Delta}$ in $\Shift(n,m)$ such
that $C_{v,\Delta}$ contains $v$ and its associated configuration
is $\Delta$. Firstly, obtain the differential sequence
$\nabla=<l_0,\cdots,l_{d-1}>$ from the necklace $\Delta$. Then the
cycle $C_{v,\Delta}$, which will be simplified as $\seq{v}$ when
$\Delta$ is clear, is given by $<v,v+k_0,v+k_1,\cdots,v+k_{d-2}>$
where $k_p=l_0+\cdots+l_p$. The cycle $\seq{0}$, which will play
an important role in the following analysis, is called the {\em
generic cycle} of $\Delta$ and its vertex set is denoted by
$V_\Delta$. Now Theorem~\ref{thm:main2} can be restated as the
following one.
\begin{theorem}\label{thm:main2restate}
For the regular configuration $\Delta$ in $\config(a,b)$, the
set\\
$\{\seq{0},\seq{(m-1)},\cdots,\seq{(k-1)(m-1)}\}$, where
$\seq{i}=C_{i,\Delta}$, is a collection of pairwise disjoint
cycles in $\Shift(n,m)$.
\end{theorem}

Notice that we can easily verify the above theorem for $a=0$ or
$b=0$. In fact, in this case the regular configuration in
$\config(n,m)$ is exactly the regular polygon. Therefore, in the
following proof, we will assume $a>0$ and $b>0$ for simplicity. On
the other hand, we do not give explicitly the algorithm stated in
Theorem~\ref{thm:main2}, since it can be easily constructed from
Theorem~\ref{thm:main2restate} and Algorithm I (see
Section~\ref{sect:findingconf}).  Furthermore, The above theorem
has an interesting corollary.

\begin{corollary}\label{thm:sequence}
Given $n$, there exist $k$ disjoint $d$-sequences $A_1,\cdots,A_k$
such that for each $i$, $\nabla(A_i)$ consists only of $m$ and $1$
and the number of $m$ is $a$.
\end{corollary}

Here $k$ is sharp. Notice that $[n]$ contains $k$ disjoint sets of
size $d$. In the remainder of this section, we are going to prove
Theorem~\ref{thm:main2restate}, and hence Theorem~\ref{thm:main2}.
Before that, we need some preparations. Firstly, we fix a
configuration $\Delta$ in $\config(a,b)$ with its characteristic
sequence $\{x_0,\cdots,x_{b-1}\}$.

\begin{definition}
Given a subset $B\subseteq V(C_n)$, its {\em differential set}
$D(B)$ is defined to be $\{b_i-b_j (\md n)~|~ \forall b_i,b_j\in
B\}$.
\end{definition}

\begin{proposition}\label{edges}
$\seq{i}\cap \seq{j} \not = \emptyset$ if and only if $j-i\in D(
V_\Delta )$. Here the subtracting is calculated with modulo $n$.
\end{proposition}
\begin{proof} Recall that $V_\Delta=\{0,k_0,k_1,\cdots,k_{d-2}\}$ is the vertex set of
$C_{0,\Delta}$. Then $\seq{i}\cap \seq{j} \not = \emptyset$ if and
only if there exists a pair of indices $p,q$ such that
$i+k_p=j+k_q$, which is equivalent to $j-i\in  D( V_\Delta ) $.
\end{proof}

\begin{corollary}\label{cor:edges}
$\seq{i}\cap \seq{j} \not = \emptyset$ if and only if
$\seq{i+1}\cap \seq{j+1} \not = \emptyset$.
\end{corollary}

\begin{proposition}\label{prop:diff set}
$D( V_\Delta )=\{l_i+l_{i+1}+\cdots+l_{i+s}|1\leqslant i \leqslant
d, 0\leqslant s<d \}\cup {0}$.
\end{proposition}
\begin{proof}
$\forall x,y\in V_\Delta$, if $x=y$, then $x-y=0$; otherwise we
have: $x=l_0+l_1+\cdots+l_p$ and $y=l_0+l_1+\cdots+l_q$ for some
$0<p,q<d$ where $p\not = q$. If $p>q$, the
$x-y=l_{p+1}+\cdots+l_{q}$. Otherwise from
$(l_0+l_1+\cdots+l_{d-1})=n$ we have
\begin{eqnarray*}
x-y&\equiv&x+n-y\\
&=&(l_0+l_1+\cdots+l_p)+(l_0+l_1+\cdots+l_{d-1})-(l_0+l_1+\cdots+l_{q})\\
&=&(l_0+l_1+\cdots+l_p)+(l_{q+1}+\cdots+l_{d-1})\\
&=&l_{q+1}+\cdots+l_{d-1}+l_0+l_1+\cdots+l_p .
\end{eqnarray*}
\end{proof}


For the configuration $\Delta$, recall that  $\mu_j$ is defined in
Section~\ref{sec:regular} as
$$\mu_j=\min_{0\leqslant i \leqslant
{b-1}}\{x_i+x_{i+1}+\cdots+x_{i+j-1}\},$$ for $1\leqslant j
\leqslant b$ and $\mu_{-1}=\mu_0=0$. Similarly we define $\xi_j$
as
$$\xi_j=\max_{0\leqslant i \leqslant
{b-1}}\{x_i+x_{i+1}+\cdots+x_{i+j}\}$$ for $0\leqslant j \leqslant
b-1$ and $\nu_b=a-1$. Here we let $\nu_{b-1}=a$ and $\nu_b=a-1$ to
satisfy the boundary condition in the following proposition.

\begin{proposition}\label{mainproof:1}
$D( V_\Delta )=\{p_{j}m+j~|~0\leqslant j \leqslant b,
~\mu_{j-1}\leqslant p_j \leqslant \xi_j\}$
\end{proposition}

\begin{proof}
The boundary cases can be verified directly and the other cases
are following from proposition (\ref{prop:diff set}) by
considering the number of 1s in the expressions of the elements in
$D( V_\Delta )$.
\end{proof}

 Now we are going to consider the case when $\Delta$ is the regular configuration
 in $\Shift(n,m)$. The following theorem is a
main step in the proof of Theorem~\ref{thm:main2}.

\begin{theorem}\label{maintheorem}
For the regular configuration $\Delta$ in $\config(a,b)$,
$\seq{0}~\cap~\seq{q(m-1)}=\emptyset$ for $1\leqslant q \leqslant
k-1$.
\end{theorem}

\begin{proof}
We will prove the theorem by contradiction. By assumption,
$\seq{i}$ is the cycle $C_{i,\Delta}$ for the regular
configuration $\Delta$ in $\config(n,m)$. If the theorem fails,
then there exists one integer $q\in [1,k-1]$ such that
$\seq{0}~\cap~ \seq{q(m-1)}\not =\emptyset$.\\

From Proposition~\ref{edges}, $q(m-1)\in D(V_\Delta)$. Together
with Proposition~\ref{mainproof:1}, this implies the following
equation has a solution for the variables $j$ and $q$ such that
$0\leqslant j \leqslant b$ and $1\leqslant q \leqslant k-1$.
\begin{equation}\label{eq: mainstep:1}
q(m-1)\equiv p_{j}m+j \ (\md n).
\end{equation}

Let $$r=\floor{\frac{q(m-1)}{am+b}}.$$ Since $p_{j}m+j<n$ from the
definition, Equation~(\ref{eq: mainstep:1}) can be simplified as:
\begin{equation}
q(m-1)  = p_{j}m+j+r(am+b).
\end{equation}

The above equation can be further simplified as:
\begin{equation}\label{mainequation}
m=\frac{q+rb+j}{q-p_j-ra}
\end{equation}

Now we are going to deduce a contradiction from the assumption
that the above equation has an integer solution for the variables
$j$ and $q$. To this end, we use the following claim, which will
be proven
later as Proposition~\ref{prop:extramain}.\\

{\bf Claim:} $q+rb+j<2m$ for $0\leqslant j \leqslant b, \
1\leqslant q \leqslant k-1$.\\

By this claim,  Equation~(\ref{mainequation}) has integer
solutions if and only if the following two equations have integer
solutions:
\begin{eqnarray}
q+rb+j=m  \label{eq: m} \\
q-p_j-ra=1  \label{eq: s}
\end{eqnarray}

By eliminating $q$ from the above equations we have:
\begin{equation}\label{eq:ms}
m=1+rb+j+p_j+ra=1+r(a+b)+j+p_j.
\end{equation}

On the other hand, we can solve $q$ from Equation~(\ref{eq: m}) to
get: $q=m-rb-j $. Together with $q\leq k-1$, we have:
$$k\geq 1+m-rb-j.$$

Since $$k=\floor{\frac{am+b}{a+b}},$$ we know,
\begin{eqnarray*}
& k(a+b) &\leqslant am+b \\
 \Rightarrow& (1+m-rb-j)(a+b) &\leqslant am+b \\
  \Rightarrow& a+bm &\leqslant rb(a+b)+j(a+b) \\
      \Rightarrow&       m    &\leqslant r(a+b)+(1+\frac{a}{b})j-\frac{a}{b}.
\end{eqnarray*}

Together with Equation(\ref{eq:ms}):
\begin{eqnarray*}
& 1+r(a+b)+j+p_j &\leqslant
         r(a+b)+(1+\frac{a}{b})j-\frac{a}{b} \\
         \Rightarrow& 1+p_j &\leqslant \frac{a}{b}(j-1) \\
         \Rightarrow& 1+\mu_{j-1} &\leqslant \frac{a}{b}(j-1)
\end{eqnarray*}
where the last step comes from the fact that $\mu_{j-1} \leqslant
p_j$. Therefore if the theorem fails, then there must exist an $j
\in [0,b]$ such that $1+\mu_{j-1} \leqslant \frac{a}{b}(j-1)$. But
from Lemma~\ref{lem:regsecond}, $1+\mu_{j-1}> \frac{a}{b}(j-1)$
for all $0\leqslant j \leqslant b$ since $\Delta$ is a regular
configuration. Thus we get an contradiction, which completes the
proof of this theorem under the assumption of the claim.
\end{proof}

Now we are going to prove the claim to complete the proof of
Theorem~\ref{maintheorem}.

\begin{proposition}\label{prop:extramain}
$q+rb+j<2m$ for $0\leqslant j \leqslant b, \ 1\leqslant q
\leqslant k-1$.
\end{proposition}

\begin{proof}
We will prove it by contradiction. If not, we have $ q+rb+j
\geqslant 2m$. Together with the assumption $j\leq b$, it implies
$q+(r+1)b \geqslant 2m$. Since $q\leqslant k-1$ and
$$r=\floor{\frac{q(m-1)}{am+b}},$$ we have:
\begin{equation}
  (k-1)+(1+\frac{(k-1)(m-1)}{am+b} \geqslant 2m
\end{equation}
By using the fact that $$k=\floor{\frac{am+b}{a+b}},$$ we get:
\begin{equation}
(\frac{am+b}{a+b}-1)+(1+\frac{(\frac{am+b}{a+b}-1)(m-1)}{am+b})b
    \geqslant 2m.
\end{equation}
It can be further simplified as:
\begin{equation}
a(a+b)bm+(ab+b^2)b  \geqslant a(a+b)m^2+(a+b)(a+2b)m.
\end{equation}
By dividing $a+b$ from both sides, we obtain:
\begin{eqnarray*}
       & abm+b^2\geqslant am^2+am+2bm \\
        \Rightarrow& b^2\geqslant (am+a+2b-ab)m\\
         \Rightarrow& b^2\geqslant (a+2b)m.
\end{eqnarray*}
This is a contradiction since $a\geqslant 0$, $0 \leqslant b<m$
and $a+b\not = 0$.
\end{proof}

The last step in this section is to prove
Theorem~\ref{thm:main2restate} with Theorem~\ref{maintheorem},
which is relatively easy.\\

{\em The proof of Theorem~\ref{thm:main2restate}} If the theorem
fails, then two cycles in $\mathcal{C}$ are not disjoint, which
means $\seq{0}~\cap~\seq{q(m-1)}\not= \emptyset$ for some $1\leq q
\leq k-1$, a contradiction to Theorem~\ref{maintheorem} since
$\Delta$ is a regular configuration.

\section{Conclusions}
In this paper, we give a relatively new definition of regular
configurations, which unifies the ``regularity" defined in many
models. Some properties of regular configurations are discussed,
including their constructions and the symmetry groups.

Regular configurations, or balanced words as it called in symbolic
dynamics (see Section~\ref{sec:words}), are showed to optimize a
number of quantities
 in words and ergodic theory~\cite{Jenkinson}.
 In this paper we extend this to graph theory. They are used to obtain a bound
 of the cycles packing number for shift graphs. In a forthcoming
 paper~\cite{Wu}, a polynomial algorithm is proposed to calculate
 $\nu_0(\Shift(n,m))$ while to calculate $\nu_0(G)$ is {\bf
 NP}-hard for general graph $G$.

A model not covered in this paper is the Kawasaki-Ising model,
which is studied in~\cite{CamWu}. There regular configurations are
characterized as the ground states in this model.

\section{Acknowledgements}

I would like to express my special thanks to Peter J Cameron for
his encouragement and help during this research, and to S\o ren
Riss for turning my attention to study the guessing number of
shift graphs~\cite{CamRiisWu}, which motivates this research. I am
also grateful to Oliver Jenkinson for his helpful conversation on
symbolic dynamics and kindly giving me the
preprint~\cite{Jenkinson}, Penwei Hao for the discussion on
computer graphics and Lihu Xu for turning my attention to the
Kawasaki-Ising model.

\end{document}